\documentclass[reqno,12pt,letterpaper]{amsart}
\usepackage[proof]{jwmacros}
\usepackage{bbm}

\newcommand{\bv}{\mathbf{v}}
\newcommand{\scd}{q}
\newcommand{\bscd}{\textbf\scd}

\newtheorem{theorem}{Theorem}[section]

\newtheorem*{theorema}{Theorem A}
\newtheorem*{theoremb}{Theorem B}

\newtheorem{non-theorem}[theorem]{Non-Theorem}

\theoremstyle{remark}
\newtheorem{remark}[theorem]{Remark}

\renewcommand{\epsilon}{\varepsilon}

\newcommand{\ep}{\epsilon}

\newcommand{\abs}[1]{{\lvert{#1}\rvert}}

\newcommand{\pa}{\partial}

\newcommand{\vl}{ \mathbf{v}_{\mathrm{L}} }
\newcommand{\vr}{ \mathbf{v}_{\mathrm{R}} }
\newcommand{\jl}{ \mathcal{J}_{\mathrm{L}} }
\newcommand{\jr}{ \mathcal{J}_{\mathrm{R}} }

\newcommand{\Lap}{\Delta}

\begin{document}

\title[subcritical channels]{Internal waves in a 2D subcritical channel}

\author{Zhenhao Li}
\email{zhenhao@mit.edu}
\address{Department of Mathematics, Massachusetts Institute of Technology, Cambridge, MA 02139}
\author{Jian Wang}
\email{wangjian@ihes.fr}
\address{Institut des Hautes {\'E}tudes Scientifiques, Bures-sur-Yvette, France, 91893}
\author{Jared Wunsch}
\email{jwunsch@math.northwestern.edu}
\address{Department of Mathematics, Northwestern University, Evanston, IL 60208}

\begin{abstract}
We analyze the scattering of linear internal waves in a two dimensional channel with subcritical bottom topography. We construct the scattering matrix for the internal wave problem in a channel with straight ends, mapping incoming data to outgoing data; this operator turns out to differ by a smoothing operator from the pullback by the ``bounce map'' for boundary data obtained by ray-tracing.  As a consequence we obtain unique solvability of the inhomogeneous stationary scattering problem subject to an appropriate outgoing radiation condition.
\end{abstract}

\maketitle


\begin{figure}
    \centering
    \includegraphics[width=0.8\textwidth]{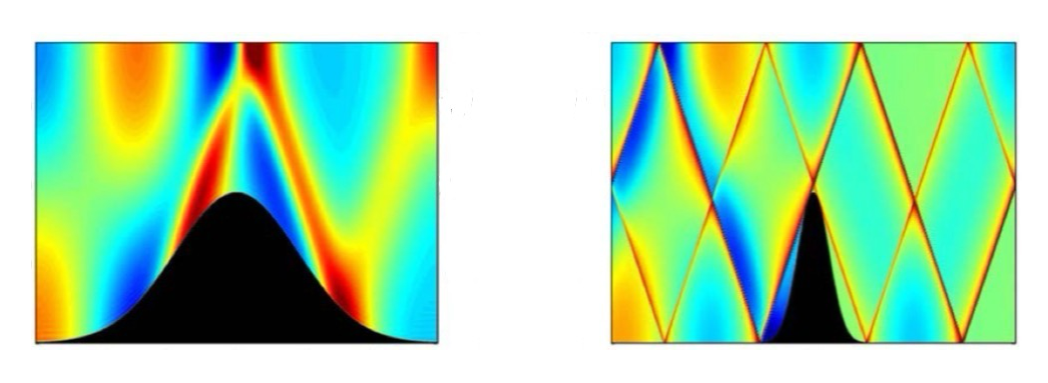}
    \caption{Scattering of a low-frequency incoming wave (traveling from left to right in both figures) by smooth bottom bumps (black). Colors represent the velocity of the internal waves. Left: the topography is subcritical (see Definition~\ref{def:subcritical} below). Right: the topography is supercritical. Figure from \cite{MaCaPe:14} (reproduced with permission).}
    \label{fig:topo}
\end{figure}

\section{Introduction}\label{s-intro}

Linear internal waves with forcing in a 2D domain $\Omega$ are described by the following Poincar\'e equation with a Dirichlet boundary condition:
\begin{equation}\label{poincare}(\partial_t^2\Delta+\partial_{x_2}^2)\psi(t,x)=F(t,x), \ \psi|_{\partial\Omega}=0, \ \psi|_{t=0}=\partial_t \psi|_{t=0}=0.
\end{equation}
Here $F(t,x)$ is a forcing term and $\psi$ is the stream function of the fluid such that the velocity of the fluid is given by $(\partial_{x_2} \psi, -\partial_{x_1}\psi )$. For the derivation of \eqref{poincare}, we refer to \cite{sobolev64, Ra:73, mbsl97, brouzet16, djov18, cdvsr20}. The evolution of internal waves in a bounded domain with periodic forcing has been the focus of considerable recent interest \cite{DyWaZw:21, Li:23, CoLi:24, Li:24}.

Here we study 2D internal waves in an \emph{unbounded} channel with flat horizontal ends. That is, we consider 
\begin{equation*}
    \Omega:=\{x\in \mathbb R^2 \mid  G(x_1)<x_2<0 \} \text{ with } G\in C^{\infty}(\mathbb R;\RR), \ G<0, \  G|_{\mathbb R\setminus [-R_0, R_0]}=-\pi
\end{equation*}
for some $R_0>0$.

Formal Fourier--Laplace transform of \eqref{poincare} in time, with zero forcing, yields the stationary equation
$$
(-\lambda^2\Lap + \partial_{x_2}^2) \hat \psi(\lambda, x) =0, \ \hat\psi(\lambda,x)|_{\pa \Omega}=0,
$$
which we rewrite as
\begin{equation}\label{stationary}
    P(\lambda)u_{\lambda}=0, \ u_{\lambda}|_{\partial\Omega}=0
\end{equation}
with
\begin{equation*}
     P(\lambda):=-\lambda^2\partial_{x_1}^2+(1-\lambda^2)\partial_{x_2}^2,
\end{equation*}
and where $u_\lambda(x) =\hat\psi(\lambda,x).$
Note that for $\lambda \in (0,1)$ this is a \emph{hyperbolic} equation, hence the usual results of stationary scattering theory, usually formulated in the context of elliptic operators with a spectral parameter (e.g., the Helmholtz equation), do not apply.  In this note, we nonetheless pursue the basic constructions of stationary scattering theory in the hyperbolic context.

In the flat ends of the channel, it is easy to solve \eqref{stationary} exactly, and we show below
(Definition~\ref{def:io}) that we may split the solution into \emph{incoming} and \emph{outgoing} parts in that region; this decomposition is phrased in terms of the Neumann data of the solution along the flat upper boundary component. 
Between the flat ends, the boundary of course varies and we have no closed-form solution.  If we view \eqref{stationary} as a wave equation, this corresponds to a boundary of a one-dimensional region that evolves in time.  An important hypothesis in what follows is that the boundary should not move faster than the speed of propagation; such a boundary is called \emph{subcritical} (see Definition~\ref{def:subcritical} below).

Our main result, stated more precisely below as Theorem~\ref{t-matrix}, is as follows:
\begin{theorema}
Suppose $\Omega$ is subcritical for a given $\lambda\in (0,1)$.
Incoming scattering data $\mathbf \scd^i \in L^2$ determines a unique solution to \eqref{stationary}, with outgoing scattering data $$\mathbf \scd^o=: \mathbf S(\lambda) \mathbf \scd^i.$$ The operator $\mathbf S(\lambda)$ (the scattering matrix) is given modulo a smoothing operator by pullback along the multiple bounce map $b$ given by ray-tracing along characteristics.
\end{theorema}
We moreover show below that $\mathbf S(\lambda)$ is unitary on an appropriately chosen Sobolev space (see Theorem~\ref{t-matrix} for details).

An important consequence of this theorem is the following corollary concerning the associated inhomogeneous problem; a complete statement of this result is below in Theorem~\ref{t-outgoing}.
\begin{theoremb}
Suppose $\Omega$ is subcritical for a given $\lambda\in (0,1)$.
Let $f(x) \in L^2_{\mathrm{comp}}(\Omega).$ There exists a unique \emph{outgoing} solution $u$ to the inhomogeneous equation 
$$
P(\lambda) u(x)=f(x).
$$
\end{theoremb}
The outgoing condition thus plays the role of the usual Sommerfeld radiation condition in conventional Euclidean scattering theory. 

We now describe in more detail the geometry of characteristics for the equation $P(\lambda)u=0$, as well as the outgoing condition.
For $\lambda\in (0,1)$, the characteristic lines of $P(\lambda)$ are level sets of $\ell_{\lambda}^{\pm}: \Omega\to \RR$ where
\begin{equation}\label{eq:char_lines}
    \ell_{\lambda}^{\pm}(x):=\pm \frac{ x_1 }{\lambda} +\frac{x_2}{\sqrt{1-\lambda^2}}.
\end{equation}
These characteristic lines have constant slopes $\pm c(\lambda)$ with
\begin{equation}\label{eq:c_def}
    c(\lambda):=\frac{\sqrt{1-\lambda^2}}{\lambda}.
\end{equation}

\begin{defi}\label{def:subcritical}
    We say a channel $\Omega$ is {\em subcritical} for time frequency $\lambda$ if $\max|G'|<c(\lambda)$; we say $\Omega$ is {\em supercritical} for $\lambda$ if $\max|G'|>c(\lambda)$.
\end{defi}
Subcriticality is an \emph{open} condition: if $\lambda \in (0, 1)$ is subcritical, then there exists an open interval $\mathcal I \subset (0, 1)$ containing $\lambda$ such that $\Omega$ is subcritical with respect to $\lambda'$ for all $\lambda' \in \mathcal I$. 

If $\Omega$ is subcritical for $\lambda$, then each characteristic line of $P(\lambda)$ intersects each of the upper domain boundary $$\partial\Omega_{\uparrow}:=\{(x_1,0) \mid x_1\in \RR\}$$ and the lower domain boundary $$\partial\Omega_{\downarrow}:=\{(x_1, G(x_1)) \mid x_1\in \RR\}$$ precisely once. 
Therefore, there exist unique involutions $\gamma^\pm_{\lambda}: \partial \Omega \to \partial \Omega$ that satisfy
\begin{equation}\label{e-gamma_def}
    \ell_{\lambda}^\pm(x) = \ell_{\lambda}^\pm(\gamma_{\lambda}^\pm(x)), \ \gamma_{\lambda}^\pm (\partial \Omega_\uparrow) = \partial \Omega_\downarrow.
\end{equation} 
Composing the two involutions, we define the single bounce chess billiard map 
\begin{equation}\label{chess1}
    b_{\lambda}:= \gamma_{\lambda}^- \circ \gamma_{\lambda}^+ : \partial \Omega_\uparrow \to \partial \Omega_\uparrow.
\end{equation}
See Figure~\ref{fig:bounce}, and \cite{NoTr:22} for an explanation of the terminology and its history. In the following we usually identify $\partial\Omega_{\uparrow}$ with $\RR$ through $(x_1,0)\mapsto x_1$. Then $b_{\lambda}$ can be regarded as an orientation preserving diffeomorphism on $\RR$. Let $M > R + 3\pi/c(\lambda)$. Then a direct computation shows that 
\[ b_{\lambda}(x_1) = x_1+\frac{2\pi}{c(\lambda)} \quad \text{when} \quad \ |x_1|\geq M. \]
Moreover, there exist open intervals $\mathcal J_{\mathrm L}, \mathcal J_{\mathrm R}\subset \partial\Omega_{\uparrow}$ such that 
\begin{equation}\label{eq:J_def}
    \mathcal J_{\mathrm L}\subset (-\infty, -M), \ \mathcal J_{\mathrm R}\subset (M,\infty), \ |\mathcal J_{\mathrm L}|=|\mathcal J_{\mathrm R}|=\frac{2\pi}{c(\lambda)}, \ \mathcal J_{\mathrm R}=b_\lambda^N(\mathcal J_{\mathrm L}), \ N\in \mathbb N.
\end{equation}
In the following we fix $\mathcal J_{\mathrm L}$, $\mathcal J_{\mathrm R}$ and call them {\em left fundamental interval} and {\em right fundamental interval}  respectively. Later we identify $\mathcal J_{\mathrm L}$, $\mathcal J_{\mathrm{R}}$ with the torus $\mathbb T_{\lambda}:=\RR/(\frac{2\pi}{c(\lambda)}\mathbb Z)$. We also denote 
\[ \mathbf b_{\lambda}:=b_{\lambda}^N: \partial\Omega_{\uparrow}\to \partial\Omega_{\uparrow} \]
where $N$ is the same as in~\eqref{eq:J_def}, and call it the {\em multi-bounce chess billiard} map. Clearly $\mathbf b_\lambda : \jl \to \jr$. 
\begin{figure}[t]
    \centering
    \includegraphics{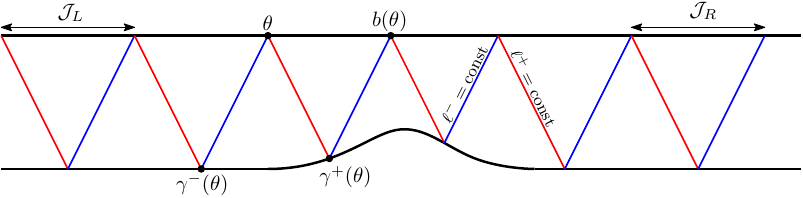}
    \caption{Diagram of $\Omega$. Level lines of $\ell^+_\lambda$ are in red and level lines of $\ell^-_\lambda$ are in blue. For a point $\theta \in \partial \Omega_\uparrow$, the location of $\gamma^\pm(\theta)$ and $b_\lambda(\theta)$ are indicated. A choice of fundamental intervals $\jl$ and $\jr$ is also labeled.}
    \label{fig:bounce}
\end{figure}

We consider solutions to the inhomogeneous equation
\begin{equation}\label{inhomogeneous}\begin{aligned}
    P(\lambda)u(x)&=f(x), \quad u|_{\partial\Omega}=0, \quad \lambda\in (0,1) \text{ subcritical},\\
    &f\in L^2_{\comp},\ \supp f\subset \overline{\Omega}\cap \{|x_1|\leq M\}.
\end{aligned}
\end{equation}
Here $L^2_{\comp}(\Omega)$ is the space of compactly supported $L^2$ functions on $\overline{\Omega}$. Note that we may always choose $M$ sufficiently large so that this is the same $M$ as in the definition of $\mathcal J_{\mathrm L}$ and $\mathcal J_{\mathrm R}$ in~\eqref{eq:J_def}. 

On the flat ends of the channel ($\abs{x_1}\gg 0$), $f$ vanishes, and solutions to \eqref{inhomogeneous} can be expanded as Fourier sine series in $x_2$, with solutions arising as superpositions of single vertical mode solutions
$$
u(x_1,x_2) = \sum_{k=1}^\infty \big( A_k^+ e^{ik c(\lambda)x_1 }+A_k^- e^{-ikc(\lambda)x_1 }\big)\sin (k x_2)
$$
with $c(\lambda)$ given by \eqref{eq:c_def}.

The bookkeeping of the incoming and outgoing parts of the solution is most easily phrased in terms of the 
Neumann data 
$$
\pa_{x_2} u\rvert_{x_2=0}= \sum_{k=1}^\infty  k A_k^+ e^{ik c(\lambda)x_1 }+ kA_k^- e^{-ikc(\lambda)x_1 }
$$
on $\mathcal J_{\mathrm L}$ and $\mathcal J_{\mathrm R},$ which we may regard as a periodic function of mean zero on the torus, or, if we prefer, as an exact one-form on the circle (after multiplication by $\mathrm d x_1$).  Let $\Pi^\pm$ be the projection map onto the positive ($+$) or negative ($-$) Fourier modes of such a function (so $\Id=\Pi^+ + \Pi^-$).
\begin{defi}\label{def:io}
    A solution $u$ to \eqref{inhomogeneous} is called {\em incoming} if 
\[ \Pi^-(\partial_{x_2}u|_{\mathcal J_{\mathrm L}}) = \Pi^+( \partial_{x_2}u|_{\mathcal J_{\mathrm R}} )=0. \]
and \emph{outgoing} if 
\[ \Pi^+(\partial_{x_2}u|_{\mathcal J_{\mathrm L}}) = \Pi^-( \partial_{x_2}u|_{\mathcal J_{\mathrm R}} )=0.\]
More generally, splitting
$$
\partial_{x_2}u|_{\mathcal J_{\mathrm{L}}}=\Pi^+ \scd^i + \Pi^-\scd^o, \ \partial_{x_2}u|_{\mathcal J_{\mathrm R}}=\Pi^- \scd^i+\Pi^+ \scd^o,
$$
we refer to $\scd^i$ resp.\ $\scd^o$ as the \emph{incoming} resp.\ \emph{outgoing} data of the solution.
\end{defi}

Our main result is that for given incoming data, there exists a unique solution to the homogeneous equation \eqref{stationary}, hence the outgoing data is uniquely specified.  The scattering matrix is the operator that maps the incoming data to the outgoing data.
More precisely, let $\mathring L^2(\mathbb T_{\lambda})$ consist of $L^2$ functions on the torus $\mathbb T_{\lambda}$ with zero mean value.
 We denote $\mathring{H}^{-\frac12}(\mathbb T_{\lambda}; T^*\mathbb T_{\lambda})$ the homogeneous Sobolev space of order $-\frac12$, consisting of one-forms of mean zero with norms defined by
 \[ \|\mathbf \scd\|_{\mathring{H}^{-\frac12}(\mathbb T_{\lambda})} = \sum_{k\in \mathbb Z}|k|^{-1}|\widehat{\mathbf \scd}(k)|^2, \ \widehat{\mathbf \scd}(k):=\frac{c(\lambda)}{2\pi}\int_{\mathbb  T_{\lambda}}e^{-ic(\lambda)k\theta}\mathbf \scd(\theta). \]
 (See \S\ref{sec:sobolev} for a brief discussion of the notation used for all the Sobolev spaces used in this paper.) The precise statement of our main theorem is then the following.

\begin{theorem}\label{t-matrix}
    Suppose $\Omega$ is subcritical for $\lambda\in (0,1)$. Then for any $  \scd^i\in \mathring L^2(\mathbb T_{\lambda})$, there exist unique $  \scd^o\in \mathring L^2(\mathbb T_{\lambda})$ and $u\in \dot H^1_{\mathrm{loc}}(\Omega)$ such that 
    \[ P(\lambda)u=0, \ \frac{-2}{c(\lambda)}\partial_{x_2}u|_{\mathcal J_{\mathrm{L}}}=\Pi^+  \scd^i + \Pi^-  \scd^o, \ \frac{-2}{c(\lambda)}\partial_{x_2}u|_{\mathcal J_{\mathrm R}}=\Pi^-  \scd^i+\Pi^+  \scd^o. \]
    The resulting map 
    \[ \mathbf S(\lambda): \mathring L^2(\mathbb T_{\lambda}; T^*\mathbb T_{\lambda}) \to \mathring L^2(\mathbb T_{\lambda}; T^*\mathbb T_{\lambda}), \quad   \scd^i \mathrm d x_1 \mapsto   \scd^o \mathrm d x_1,  \]
    is called the scattering matrix for $P(\lambda)$ in $\Omega$. Moreover, there exists a smoothing operator $R: \mathcal D'(\mathbb T_{\lambda}; T^*\mathbb T_{\lambda})\to C^{\infty}(\mathbb T_{\lambda}; T^*\mathbb T_{\lambda})$ such that 
    \[ \mathbf S(\lambda)=\Pi^+ \mathbf b_{\lambda}^* \Pi^+ + \Pi^-( \mathbf b_{\lambda}^{-1})^* \Pi^- + R. \]
    Furthermore, $\mathbf S(\lambda)$ has the improved mapping property
    \begin{equation}\label{eq:S_map_prop}
        \mathbf S(\lambda): \mathring H^s(\mathbb T_{\lambda}; T^*\mathbb T_{\lambda}) \to \mathring H^s(\mathbb T_{\lambda}; T^*\mathbb T_{\lambda})
    \end{equation}
    {for all $s\in\RR$}
    and is in fact a unitary operator on $\mathring{H}^{-\frac12}(\mathbb T_{\lambda}; T^*\mathbb T_{\lambda})$.
\end{theorem}
(The normalizing constant $-2/c(\lambda)$ is explained below in Remark~\ref{rem:neumann}.)

Using the scattering matrix constructed in Theorem \ref{t-matrix}, one can find purely {\em outgoing} solutions to the inhomogeneous stationary problem \eqref{inhomogeneous}.

\begin{theorem}\label{t-outgoing}
    Suppose $\Omega$ is subcritical for   $\lambda\in (0,1)$. Then there exists a map
    \[ \mathcal{R}(\lambda): L^2_{\mathrm{comp}}(\overline{\Omega})\to \dot H^1_{\mathrm{loc}}(\Omega) \]
    such that for any $f\in L^2_{\mathrm{comp}}(\overline{\Omega})$, the function $u:=\mathcal{R}(\lambda)f$ is the unique solution to \eqref{inhomogeneous} satisfying 
    \[ \Pi^+(\partial_{x_2}u|_{\mathcal J_{\mathrm L}})=\Pi^-( \partial_{x_2} u|_{\mathcal J_{\mathrm R}} )=0. \]
 \end{theorem}

\subsection{Relation to the oceanographic literature}
The oceanographic literature contains some explorations of the scattering problem as discussed here (it is of considerable importance, e.g., in the study of mixing in the ocean \cite{MuLi:00b}).  Longuet-Higgins \cite{Lo:69}, for example, considers the approximation to the scattering given by ray-tracing, as motivated by WKB solutions.  This was clearly understood as a high-frequency approximation: M\"uller--Liu \cite[\S 5c]{MuLi:00a} note that ``One expects reflection theory to do the worst for low incident modenumbers. This is indeed the case.''  Indeed, Baines \cite{Ba:71} performed a more refined analysis of plane-wave scattering that involved a Fredholm integral operator correcting the ray tracing approximation, which he too noted is inaccurate, especially at low wavenumbers.  Baines worked in an ocean with no surface, however, rather than the finite channel under consideration here.  Our approach is morally similar, but involves rigorous discussion of uniqueness of outgoing solutions.  Our results on the scattering matrix quantitatively justify the assertion that the reflection theory approximates the scattering matrix, by showing that the error in this approximation is rapidly decaying in the wavenumber parameter (this is the frequency-domain manifestation of the fact that the remainder term $R$ is smoothing).

We remark that a previous version of this manuscript contained an account of the limiting absorption principle, concerning the convergence of $\lim_{\ep \downarrow 0} P(\lambda-i\ep)^{-1}$ to the outgoing resolvent described above, together with consequences for the long-time asymptotics in the time-dependent problem with periodic forcing.  We have, however, discovered an error in this work, and anticipate returning to the question of the limiting absorption principle in a future work, employing a parametrix based on the bounce map similar to that in \cite{DyWaZw:21}.

\subsection{Some Sobolev spaces}\label{sec:sobolev}
Before moving on, we fix the notation for various Sobolev spaces on manifolds with boundary. If $F \subset \mathcal D'(\RR^2)$ is a closed linear subspace of Schwartz distributions, we denote by $\dot F(\Omega) \subset F$ the subspace of $F$ supported on $\overline \Omega$, and $\bar F(\Omega) = F/\dot F(\RR^2 \setminus \Omega)$ the space of extendable distributions on $\Omega$. For instance, $\dot H^1(\Omega)$ denotes the set of functions in $H^1(\RR^2)$ whose support lies in $\overline \Omega$. In particular, $\dot H^1(\Omega) \simeq H^1_0(\Omega)$, where $H^1_0(\Omega)$ denotes the usual space of trace-free $H^1$ functions on $\Omega$. We also remark that $\dot L^2(\Omega) = \bar L^2(\Omega) = L^2(\Omega)$. For more details, see \cite[Appendix B]{Ho:85}. We will also use the subscripts $\loc$ or $\comp$ to denote local and compactly supported Sobolev spaces respectively. Finally, we denote by $\mathring H^s(\mathbb T; T^* \mathbb T)$ to be the subset of distributional one-forms $\mathbf v \in H^s(\mathbb T; T^* \mathbb T)$ such that $\int \mathbf v = 0$. 

\vspace{5pt}
\noindent
{\bf Acknowledgments.}
The authors would like to thank Semyon Dyatlov and Carl Wunsch for helpful discussions.
J.~Wunsch acknowledges partial support from NSF grant DMS--2054424 and from Simons Foundation Grant MPS-TSM-00007464.
Z.~Li acknowledges partial support from Semyon Dyatlov's NSF grant DMS--2400090.

\section{Characteristics and Cauchy data}\label{s-stationary}

Let us start by analyzing \eqref{inhomogeneous} in the characteristic coordinates of \eqref{eq:char_lines}, i.e., in coordinates $$y_{\pm}:=\ell_{\lambda}^{\pm}(x)=\pm \frac{ x_1 }{\lambda} +\frac{x_2}{\sqrt{1-\lambda^2}}.$$ In these coordinates,
\[ P(\lambda) = \tfrac14 \partial_{y_+}\partial_{y_-}. \]
The upper boundary is given by
\[ \partial\Omega_{\uparrow}=\{ (y_+, y_-) \mid y_++y_-=0 \} \]
and we parametrize $\partial\Omega_{\uparrow}$ by 
\begin{equation*} \label{eq:yparam}
\mathbf y: \RR \to \partial\Omega_{\uparrow}, \ \theta\mapsto \left( \frac{\theta}{\sqrt{1-\lambda^2}}, -\frac{\theta}{\sqrt{1-\lambda^2}} \right). 
\end{equation*}
The lower boundary is given by
$$
\partial\Omega_{\downarrow}=\{ (y_+, y_-) \mid K(y_1,y_2)=0\}
$$
where
$$
K(y_1,y_2):= \frac{\sqrt{1-\lambda^2}}{2} (y_++y_-)-G(\lambda(y_+-y_-)/2);
$$
hence
$$
\Omega =\{K \geq 0\} \cap \{y_++y_-<0\}.
$$
See Figure~\ref{fig:ypm} for a diagram of the domain in $y_\pm$ coordinates. Subcriticality implies that there exists $\ep_G>0$ such that
$$
\pa_{y_1}K\leq -\ep_G, \quad \pa_{y_2} K\leq -\ep_G
$$
hence
\begin{equation}\label{lowerboundary}
(y_1,y_2) \in \pa \Omega_{\downarrow},\ y_1'< y_1,\ y_2'<y_2 \Longrightarrow (y_1',y_2') \notin \Omega,
\end{equation}
and more generally there exists $R_G>0$ such that
\begin{equation}\label{lowerboundary2}
(y_1,y_2) \in \Omega \Longrightarrow \{(y_1',y_2')\mid y_i'\leq y_i\}\cap \Omega\subset B^2((y_1,y_2), R_G).
\end{equation}

Returning to \eqref{chess1}, we see that under this parametrization there exists $M>0$ depending on the topography $G$ such that 
\[ b_{\lambda}(\theta) = \theta+2\pi \text{ for all } |\theta|\geq M \]
We use $\mathcal J_{\pm}$ to denote the pre-image of the left/right fundamental intervals defined in \S \ref{s-intro} when there is no ambiguity.

\subsection{Reduction to the boundary}\label{sec:compact}
In $(y_+, y_-)$ coordinates, \eqref{inhomogeneous} becomes
\[ \partial_{y_+}\partial_{y_-}u = 4f, \ u|_{\partial\Omega}=0. \]
Let
\begin{equation}\label{e-u0}
U_0(y_+, y_-):= 4\int_{-\infty}^{y_+}\int_{-\infty}^{y_-} f(s_+, s_-) \mathrm d s_+ \mathrm d s_-. 
\end{equation}
\begin{figure}
    \centering
    \includegraphics[width=0.5\linewidth]{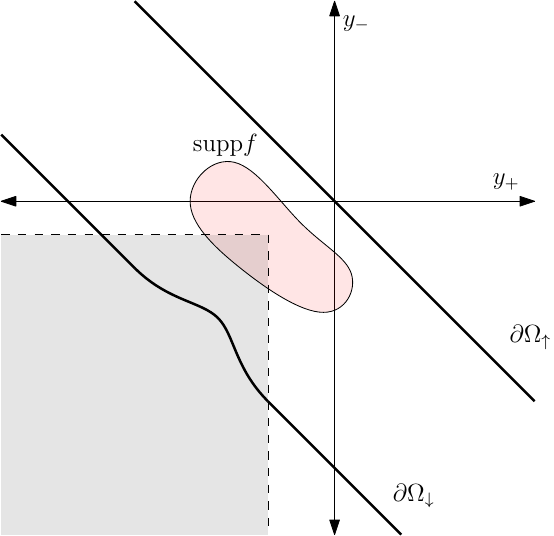}
    \caption{Domain $\Omega$ in $y_\pm$ coordinates. The support of the forcing profile $f$ is shaded and labeled, and the rectangular shaded region is the integration kernel in~\eqref{e-u0}.}
    \label{fig:ypm}
\end{figure}
See Figure~\ref{fig:ypm} for a diagram of the integration kernel with respect to the domain and $f$. By \eqref{lowerboundary2}, the integral is in fact over a compact region for any given $(y_+,y_-)$, hence it is well-defined.  Likewise, \eqref{lowerboundary2} implies that $\supp U_0 \cap \overline \Omega$ is compact, since $f$ is compactly supported.
By the subcriticality assumption, \eqref{lowerboundary} implies that the support of $U_0$ is inside $K \geq 0$.
Note that $U_0\in H^1(\overline{\Omega})$ since $\pa_{y_i} U_0 \in L^2$, $i=1,2.$ Moreover, one can check that
$$
g:= U_0|_{\partial\Omega_{\uparrow}}\in H^1_{\mathrm{comp}}(\partial\Omega_{\uparrow})
$$
retains the same regularity as $U_0$ despite the restriction.
Thus, summing up, we have established the following.
\begin{lemm}
Let $U_0$ be given by \eqref{e-u0}.  Then
\begin{equation}\label{e-g}
P(\lambda)U_0=f, \ U_0\in H^1_{\mathrm{comp}}(\overline{\Omega}), \ U_0|_{\partial\Omega_{\downarrow}}=0, \ g=U_0|_{\partial\Omega_{\uparrow}}\in H^1_{\mathrm{comp}}(\partial\Omega_{\uparrow}). 
\end{equation}
\end{lemm}

Now to solve $u$ in \eqref{inhomogeneous}, one only needs to solve for $w:=U_0-u$ that satisfies the homogeneous boundary value equation
\begin{equation}\label{bvp} 
\partial_{y_+}\partial_{y_-} w = 0, \quad w|_{\partial \Omega_{\downarrow}} = 0, \quad w|_{\partial \Omega_{\uparrow}} = g \in H^1_{\comp} (\partial \Omega_{\uparrow})
\end{equation}
\begin{lemm}\label{lem:w_decomp}
    Suppose $\Omega$ is subcritical for $\lambda\in (0,1)$ and $w\in \dot H^1_{\mathrm{loc}}(\Omega)$ solves \eqref{bvp}. Then there exist $w_{\pm}\in \dot H^1_{\mathrm{loc}}(\RR)$ such that 
    \[ w(y_+, y_-) = w_+(y_+)+ w_-(y_-). \]
\end{lemm}
\begin{proof}
    Define $w_+^1:=\partial_{y_+}w$. Then we know 
    \[ w_+^1\in L^2_{\mathrm{loc}}(\Omega), \ \partial_{y_-}w_+^1=0. \]
    The second equation together with the assumption that $\Omega$ is subcritical shows that $w_+^1$ depends only on $y_+$. Since $y_+$ can take any value in $\RR$,  we know $w_+^1$ defines a function on $\RR$. For every $a>0$, there exists $\delta>0$ such that the parallelogram $\Omega_{a,\delta}:=\{(y_+, y_-) \mid |y_+|\leq a, \ -y_+-\delta \leq y_-\leq -y_+\}$ is a subset of $\Omega$. Thus we see that
    \[ \|w_+^1\|_{L^2([-a,a])}^2\leq \delta^{-1}\|\partial_{y_+}w\|_{L^2(\Omega_{a,\delta})}^2 \leq \delta^{-1}\|w\|_{H^1(\Omega_{a,\delta})}^2<\infty. \]
    This shows that $w_+^1\in L^2_{\mathrm{loc}}(\RR)$. Define 
    \[ w_+(y_+):=\int_0^{y_+} w_+^1(s)\mathrm d s.  \]
    Then $w_+$ satisfies 
    \[ w_+\in H^1_{\mathrm{loc}}(\RR), \ \partial_{y_+}(w(y_+, y_-)-w_+(y_+))=0. \]
    This shows that $w_-:=w-w_+$ is a function depending only on $y_-$. Moreover, $\partial_{y_-}w_- = \partial_{y_-}w$. Similar argument as above shows that $w_-\in H^1_{\mathrm{loc}}(\RR)$.
\end{proof}

By abuse of notation, we also denote the pullbacks $y_\pm^* w_\pm$ by $w_\pm$, where $y_{\pm}^*$ denote projections $(y_+, y_-)\mapsto y_{\pm}$, so that $w_\pm$ can be viewed as elements of $\bar H^1_{\loc}(\Omega)$. In this sense, $w_\pm$ can be restricted to $\partial \Omega$, and the restrictions lie in $H^1_{\mathrm{loc}}(\partial \Omega)$. Note that 
\begin{equation*}
    w_\pm|_{\partial \Omega} = (\gamma^\pm)^* (w_\pm|_{\partial \Omega})
\end{equation*}
by \eqref{e-gamma_def}.
Applying the boundary conditions in \eqref{bvp} we have 
\[ w_+|_{\partial\Omega_{\uparrow}}+w_-|_{\partial\Omega_{\uparrow}} =g, \ w_+|_{\partial\Omega_{\downarrow}} + w_-|_{\partial\Omega_{\downarrow}}=0. \]
Therefore, using the $\mathbf y$ parametrization \eqref{eq:yparam} of $\partial\Omega_{\uparrow}$, we have 
\[\begin{split} 
w_{\pm}|_{\partial\Omega_{\uparrow}}(\theta) = & w_{\pm}|_{\partial\Omega_{\downarrow}}( \gamma^{\pm}(\theta) ) = -w_{\mp}|_{\partial\Omega_{\downarrow}}(\gamma^{\pm}(\theta)) = -w_{\mp}|_{\partial\Omega_{\uparrow}}( \gamma^{\mp}\circ\gamma^{\pm}(\theta) ) \\
= & -w_{\mp}|_{\partial\Omega_{\uparrow}}( b^{\pm 1}(\theta) ) = w_{\pm}|_{\partial\Omega_{\uparrow}}( b^{\pm 1}(\theta) ) - g(b^{\pm 1}(\theta)).
\end{split}\]
That is, 
\begin{equation}\label{e-1bounce}
    w_{\pm}|_{\partial\Omega_{\uparrow}} - (b^{\pm 1})^*( w_{\pm}|_{\partial\Omega_{\uparrow}} )  = -( b^{\pm 1} )^*g.
\end{equation}
Iterate \eqref{e-1bounce} $N$ times, restrict $w_+$ to the left/right fundamental intervals $\mathcal J_{\mathrm L}$, $\mathcal J_{\mathrm R}$ (see \S \ref{s-intro}), then differentiate both sides, and we find 
\begin{equation}\label{e-transport}
    \mathbf v_{\mathrm L} - \mathbf b^* \mathbf v_{\mathrm R} = \mathbf g.
\end{equation}
Here $\mathbf b=b_\lambda^N$ is the multi-bounce chess billiard map defined in \S \ref{s-intro} (we now suppress the $\lambda$ subscript) and 
\begin{equation}\begin{split}\label{e-vpm}
\mathbf v_{\bullet}:= & \mathrm d w_+|_{\mathcal J_{\bullet}} \in \mathring L^2(\mathcal J_{\bullet}; T^*\mathcal J_{\bullet}), \quad \bullet = \mathrm L, \mathrm R,   \\
\mathbf g:= & -\sum_{k=1}^{N} (b^{ k})^* \mathrm d g|_{\mathcal J_{\mathrm L}} \in \mathring L^2(\mathcal J_{\mathrm L}; T^*\mathcal J_{\mathrm L}). 
\end{split}\end{equation}

Let us take a step back and interpret all the new objects we have defined. We claim that $\mathbf v_{\mathrm L}$ and $\mathbf v_{\mathrm R}$ are essentially the Neumann data of the solution $u$ 
on $\mathcal J_{\mathrm L}$, $\mathcal J_{\mathrm R}$ respectively.  
\begin{lem}\label{l-neumann}
    Suppose $u \in \dot H^1_{\loc}$ satisfies the stationary internal wave equation~\eqref{inhomogeneous}. Then 
    \begin{equation*}
        v_{\uparrow} := \mathbf j^* (\partial_{y_+} u \, \mathrm d y_+)
    \end{equation*}
    is well-defined in $L^2_{\mathrm{comp}}(\partial \Omega_\uparrow; T^* \partial \Omega_{\uparrow})$, where $\mathbf j : \partial \Omega_\uparrow \to \RR^2$ is the canonical embedding. Furthermore,
    \begin{equation*}
        \mathbf v_{\bullet} =  -v_\uparrow|_{\mathcal J_{\bullet}}, \ \bullet = \mathrm L, \mathrm R
    \end{equation*}
    where $\vl$ and $\vr$ are as defined in~\eqref{e-vpm}.
\end{lem}

\noindent
\begin{remark}\label{rem:neumann} It is easy to check that in the $x_1$ coordinates on $\partial \Omega$, $v_\uparrow$ is given by 
\begin{equation*}
    v_\uparrow = \tfrac{c(\lambda)}{2} \partial_{x_2} u|_{\partial \Omega_\uparrow} \mathrm{d}x_1.
\end{equation*}
Therefore, $\vl$ and $\vr$ are simply a multiple of the Neumann data.
\end{remark}
\begin{proof}
    Recall that 
    \[u = U_0 - ( w_++w_- )\]
    where $w_\pm$ depends only on $y_\pm$. Therefore, 
    \[ \partial_{y_+}u(y_+, y_-)\mathrm d y_+ = \partial_{y_+}U_0(y_+, y_-)\mathrm d y_+ - \mathrm d w_+. \]
    Since 
    $U_0 \in H^1_{\mathrm{comp}}(\overline \Omega)$, it follows that $v_\uparrow$ is well-defined in $L^2_{\mathrm{comp}}$. 
    The relationship to $\vl$ and $\vr$ follows from the fact that $U_0$ vanishes in a neighborhood of $\jl \cup \jr.$
\end{proof}

Motivated by the above lemma, we call $\vl$ and $\vr$ the Neumann data at left and right infinity respectively. Next, we retrace our steps and verify that if~\eqref{e-transport} is satisfied, we indeed have a solution with the given Neumann data. 
\begin{lem}\label{l-existence}
    Assume that $f \in \bar H^s_{\comp}$ for some $s \ge 0$. Let $\vl \in \mathring H^s(\jl; T^* \jl)$ and $\vr \in \mathring H^s(\jr; T^* \jr)$. If $\vl$ and $\vr$ satisfies \eqref{e-transport}, then there exists $u \in \dot H^1_{\loc}(\Omega) \cap \bar H^{s + 1}_{\loc}(\Omega)$ that satisfies~\eqref{inhomogeneous} with $\vl$ and $\vr$ as the Neumann data on the left and the right fundamental intervals respectively.
\end{lem}
\begin{proof}
    Let $U_0$ and $g$ be as defined in~\eqref{e-u0} and~\eqref{e-g}. Let $\jl = [\theta_0, \theta_0+2\pi)$ be the left fundamental interval defined in~\eqref{eq:J_def} using the parametrization $\mathbf y$ defined in~\eqref{eq:yparam}. Notice that $b^k(\jl)$, $k \in \mathbb Z$, tiles $\partial \Omega_{\uparrow}$. Then we can define a function $ \varpi$ on $\partial \Omega_{\uparrow}$ by 
    \begin{equation}\label{eq:omegas}
        \begin{aligned}
             \varpi|_{\jl}(\theta) := & \int_{\theta_0}^\theta \mathbf \vl, \\
             \varpi|_{b^k(\jl)} := & (b^{-k})^* ( \varpi|_{\jl}), \\
             \varpi|_{b^{-k}(\jl)}:= & (b^k)^*( \varpi|_{\jl}) - \sum_{n=1}^k (b^n)^*g|_{b^{-k}(\jl)}, \ k\geq 1.
        \end{aligned}
    \end{equation}
    One can check that $ \varpi$ satisfies 
    \begin{equation}\label{eq:cohom}  \varpi- b^* \varpi = -b^*g. \end{equation}
    Note that since $f \in \bar H^s_{\comp}$, we have $g \in H_{\comp}^{s + 1}(\partial \Omega_\uparrow)$. Combined with the assumption that $\int \vl = 0$, it follows that $ \varpi$ is in fact continuous on the circle, as well as lying in $H^{s + 1}_{\loc}(\partial \Omega_\uparrow)$. Observe that there exist unique $w_{\pm} \in \bar H^{s + 1}_{\loc}(\Omega)$ such that for $(y_+, y_-)\in \Omega$,
    \begin{equation}\label{eq:w.omega}\begin{aligned}
        w_+(y_+, y_-) &:=   \varpi(\theta), \text{ when } y_+ = y_+(\theta), \\
        w_-(y_+, y_-) &:=  g(\theta) -  \varpi(\theta), \text{ when } y_- = y_-(\theta)
    \end{aligned}\end{equation}
    where $\mathbf y(\theta)=(y_+(\theta), y_-(\theta))$.
    We claim that
    \begin{equation*}
        u:=U_0-(w_+ + w_-)
    \end{equation*}
    is our desired solution. Clearly, $P(\lambda)u = f$, $u \in \bar H^1_{\loc}(\Omega)$, and $u|_{\partial \Omega_\uparrow} = 0$. Since $$(\gamma^+)^* w_+(y_+)=w_+(y_+),$$ while $$(\gamma^+)^* w_-(y_-)=w_-(y_-\circ \gamma^+)=w_-(y_-\circ \gamma^-\circ \gamma^+)=b^* w_-(y_-),$$ the relations \eqref{eq:w.omega} together with \eqref{eq:cohom} yield
    \[ (\gamma^+)^*(u|_{\partial\Omega_{\downarrow}}) = 0-(  \varpi + ( b^*g - b^* \varpi ) )=0. \]
    Since $(\gamma^+)^2=\mathrm{Id}$, we conclude that $u|_{\partial\Omega_{\downarrow}}=0$.
    Thus we have $u \in \dot H^1_{\loc}(\Omega) \cap \bar H^{s + 1}_{\loc}(\Omega)$. Finally, by the second equation in \eqref{eq:omegas} with $k=N$, the right Neumann data of the solution $u$ is precisely given by $\vr$ satisfying the relation~\eqref{e-transport}. 
\end{proof}

\section{Scattering matrix}

We now consider the homogeneous stationary internal wave equation~\eqref{stationary}. The the left and right Neumann data defined in~\eqref{e-vpm} satisfies
\begin{equation}\label{homogpullback}
    \vl - \mathbf b^* \vr = 0.
\end{equation}
Using the parametrization $\mathbf y$ in \S \ref{s-stationary}, we can identify $\jl$ and $\jr$ with $\mathbb S^1 = \RR/2 \pi \mathbb{Z}$ so that $\vl, \vr \in \mathring L^2(\mathbb S^1; T^* \mathbb S^1)$ and $\mathbf b \in C^\infty(\mathbb S^1; \mathbb S^1)$. Then taking the positive and negative Fourier projectors $\Pi^{\pm}$, the outgoing data can be expressed as 
\begin{align}\label{scmatrix0}
    \begin{pmatrix}
        \Pi^-\vl \\ \Pi^+\vr
    \end{pmatrix}
    = & \begin{pmatrix}
        \Pi^-\mathbf b^* \vr \\ \Pi^+ \mathbf b^{-*}\vl
    \end{pmatrix} \\ \label{scmatrix0.5}
    = & \begin{pmatrix} 
        0 & \Pi^-\mathbf b^* \Pi^+ \\ \Pi^+\mathbf b^{-*} \Pi^- & 0
    \end{pmatrix}
    \begin{pmatrix}
        \Pi^-\vl \\ \Pi^+\vr
    \end{pmatrix}
    + \begin{pmatrix}
        \Pi^-\mathbf b^*\Pi^- & 0 \\ 0 & \Pi^+\mathbf b^{-*}\Pi^+
    \end{pmatrix}
    \begin{pmatrix}
        \Pi^-\vr \\ \Pi^+\vl
    \end{pmatrix}
\end{align}
where $\mathbf b^{-*}:=(\mathbf b^{-1})^*$.
We rewrite the equation as
\begin{equation}\label{scmatrix1}
    \begin{pmatrix} \mathrm{Id} & -\Pi^-\mathbf b^*\Pi^+ \\ -\Pi^+\mathbf b^{-*} \Pi^- & \mathrm{Id}  \end{pmatrix}
    \begin{pmatrix} \Pi^-\vl \\ \Pi^+\vr \end{pmatrix}
    =\begin{pmatrix} \Pi^-\mathbf b^*\Pi^- & 0 \\ 0 & \Pi^+\mathbf b^{-*}\Pi^+ \end{pmatrix}
    \begin{pmatrix} \Pi^-\vr \\ \Pi^+\vl \end{pmatrix}.
\end{equation}
Our goal is to recover the outgoing data $\begin{pmatrix} 
\Pi^-\vl \\ \Pi^+\vr \end{pmatrix}$ uniquely in terms of the incoming data $\begin{pmatrix} \Pi^-\vr \\ \Pi^+\vl \end{pmatrix}$; to do so it suffices to invert
\begin{equation}\label{Tdef}
\begin{gathered}
    \mathbf T : \Pi^- L^2(\mathbb S^1; T^* \mathbb S^1) \times \Pi^+ L^2(\mathbb S^1; T^* \mathbb S^1) \to \Pi^- L^2(\mathbb S^1; T^* \mathbb S^1) \times \Pi^+ L^2(\mathbb S^1; T^* \mathbb S^1),\\
    \mathbf T := 
    \begin{pmatrix}
        \Id & -\Pi^- \mathbf b^* \Pi^+ \\
        -\Pi^+ \mathbf b^{-*} \Pi^- & \Id
    \end{pmatrix}.
\end{gathered}
\end{equation}
To invert $\mathbf T$ we will need the following lemma. 
\begin{lem}\label{l-low_freq}
Let $\varphi$ be an orientation-preserving diffeomorphism of $\mathbb S^1$ and let $\bv \in L^2(\mathbb S^1; T^* \mathbb S^1)$. Then
\begin{align*}
\Pi^- \varphi^* \Pi^- \bv &= 0 \text{ implies }\Pi^- \bv = 0,\\
\Pi^+ \varphi^* \Pi^+ \bv &= 0 \text{ implies }\Pi^+ \bv = 0. 
\end{align*}
\end{lem}
\begin{proof}
We treat the first case; the proof of the second is virtually identical.
Assume for the sake of contradiction that $\Pi^- \bv \neq 0$ and
$$
\Pi^- \varphi^* \Pi^- \bv=0,
$$
hence
\begin{equation}\label{e-wformula}
\varphi^* \Pi^-\bv = (\Id-\Pi^-) \varphi^* \Pi^- \bv.
\end{equation}
The operator on the right-hand-side of \eqref{e-wformula} is a smoothing operator (by the calculus of wavefront sets, using the fact that $\varphi$ is orientation-preserving), hence $\varphi^* \Pi^-\bv \in C^\infty(\mathbb S^1; T^* \mathbb S^1)$.  Thus also $\Pi^-\bv \in C^\infty(\mathbb S^1; T^* \mathbb S^1).$

Note that $\int \Pi^-\bv = 0$. We can then define the function
\begin{equation*}
    w(\theta) = \int_0^\theta \Pi^- \bv \in C^\infty(\mathbb S^1).
\end{equation*}
Clearly, $\widehat w(k) := \frac{1}{2\pi}\int_{\mathbb S^1} e^{-i k \theta} w(\theta)\, d\theta = 0$ for all $k > 0$. Therefore, 
\begin{equation*}
    \mathbf F(w) := \frac{1}{i} \int_{\mathbb S^1} \overline w \mathrm dw = 2\pi\sum_{k\leq 0}k|\widehat w(k)|^2 < 0.
\end{equation*}
Note that $\mathbf F(w) = \mathbf F(\varphi^* w)$. Therefore, there exist $k_- < 0$ such that $\widehat{\varphi^* w}(k_-) \neq 0$. Since
\[\varphi^* \Pi^- v = \mathrm d\varphi^* w\]
it follows that $\widehat{\varphi^* \Pi^- v}(k_-) \neq 0$, which contradicts $\Pi^- \varphi^* \Pi^- v = 0$. Therefore we must have $\Pi^- v = 0$. 
\end{proof}

Now it follows that $\mathbf T$ is invertible. 
\begin{lem}\label{l-ker}
The nullspace of $\mathbf T$ on $\Pi^- L^2(\mathbb S^1; T^* \mathbb S^1) \times \Pi^+ L^2(\mathbb S^1; T^* \mathbb S^1)$ is trivial.
\end{lem}
\begin{proof}
    Let $\begin{pmatrix} \bv_- \\ \bv_+\end{pmatrix} \in \Pi^- L^2(\mathbb S^1; T^* \mathbb S^1) \times \Pi^+ L^2(\mathbb S^1; T^* \mathbb S^1)$ be such that $\mathbf T \begin{pmatrix} \bv_- \\ \bv_+\end{pmatrix} = 0$. Then we must have
    \begin{equation*}
        \bv_- = \Pi^- \mathbf b^{*} \Pi^+ \mathbf b^{-*} \Pi^- \bv_-.
    \end{equation*}
    Let $\bv := \mathbf b^{-*} \Pi^- \bv_-$. Then $\bv_- = \Pi^- \mathbf b^{*} \Pi^+ \bv$, from which we see that
    \begin{equation}\label{e-v}
        \Pi^-\mathbf b^{*} \bv = \Pi^-\mathbf b^*( \mathbf b^{-*}\Pi^-\bv_- ) = \Pi^-\bv_-= \Pi^- \mathbf b^{*} \Pi^+ \bv.
    \end{equation}
    Note that the zeroth Fourier coefficient of $v$ vanishes since $\mathbf b^*$ is the pullback on 1-forms, so $\bv - \Pi^+ \bv = \Pi^- \bv$. Then it follows from~\eqref{e-v} that 
    \begin{equation*}
        \Pi^- \mathbf b^{*} \Pi^- \bv = 0.
    \end{equation*}
    By Lemma~\ref{l-low_freq}, it follows that $\Pi^- \bv = 0$. In particular, this means that $\Pi^- \mathbf b^{-*} \Pi^- \bv_- = 0$. Apply Lemma \ref{l-low_freq} again, and we see that $\bv_-=\Pi^-\bv_-=0$. 
    A similar argument shows that $\bv_+ = 0$, so the nullspace is indeed trivial. 
\end{proof}

Let us now complete the proof of Theorem \ref{t-matrix}.
\begin{proof}
    Suppose $\scd^i\in \mathring L^2(\mathbb T_{\lambda})$. We regard $\scd^i$ as an element in $\Pi^- L^2(\mathbb S^1;T^*\mathbb S^1)\times \Pi^+L^2(\mathbb S^1; T^*\mathbb S^1)$ through 
    \[\begin{split} 
    \scd^i \mapsto  ( \mathbf y^*( \Pi^-\scd^i \mathrm d x_1), \mathbf y^*( \Pi^+\scd^i \mathrm d x_1 ) )=:(\mathbf \scd^i_{\mathrm R}, \mathbf \scd^i_{\mathrm{L}}). 
    \end{split}\]
    Note then that
    $$
    \Pi^-\bscd_L^i=0,\quad \Pi^+\bscd_R^o=0.
    $$
    By Lemma \ref{l-ker}, we can define $(\mathbf \scd^o_{\mathrm{L}}, \mathbf \scd^o_{\mathrm{R}})$ and $\mathbf \scd^o$ such that 
    \begin{equation}\label{qlr}
        \begin{pmatrix} \mathbf \scd^o_{\mathrm L} \\ \mathbf \scd^o_{\mathrm R} \end{pmatrix} :=\mathbf T^{-1}\begin{pmatrix} \Pi^-\mathbf b^*\Pi^- & 0 \\ 0 & \Pi^+\mathbf b^{-*}\Pi^+ \end{pmatrix}\begin{pmatrix} \mathbf \scd^i_{\mathrm R} \\ \mathbf \scd^i_{\mathrm L} \end{pmatrix}, \ \mathbf \scd^o:=\mathbf \scd^o_{\mathrm L} + \mathbf \scd^o_{\mathrm R}.
    \end{equation}
    Define
    \[ \vl:=\mathbf \scd^i_{\mathrm L}+\mathbf \scd^o_{\mathrm L}, \ \vr:=\mathbf \scd^i_{\mathrm R} + \mathbf \scd^o_{\mathrm R}. \]
    Since $$\mathbf T  \begin{pmatrix} \mathbf \scd^o_{\mathrm L} \\ \mathbf \scd^o_{\mathrm R} \end{pmatrix}=\begin{pmatrix} \Pi^-(\bullet)\\ \Pi^+(\bullet) \end{pmatrix}, 
    $$
    examination of \eqref{Tdef} shows that 
    $$
    \Pi^+ \bscd^o_L=0,\quad \Pi^- \bscd^o_R=0,
    $$
    hence
    $$
    \begin{pmatrix} \Pi^-\vl \\ \Pi^+\vr \end{pmatrix}=\begin{pmatrix} \mathbf \scd^o_{\mathrm L} \\ \mathbf \scd^o_{\mathrm R} \end{pmatrix},\quad
\begin{pmatrix} \Pi^-\vr \\ \Pi^+\vl \end{pmatrix}
=\begin{pmatrix} \mathbf \scd^i_{\mathrm R} \\ \mathbf \scd^i_{\mathrm L} \end{pmatrix}.
    $$
Hence \eqref{qlr} implies that $\bv_{\mathrm L}$, $\bv_{\mathrm R}$ satisfy \eqref{scmatrix1} and hence also \eqref{scmatrix0}.  This yields the two relationships
$$
\Pi^-(\vl-\mathbf b^* \vr)=0,\quad \Pi^+(\vr-\mathbf b^{-*} \vl)=0.
$$
Setting $$\bv=\vl-\mathbf b^* \vr$$ and rewriting the latter equation gives
$$
\Pi^- \bv=0,\quad \Pi^+ \mathbf b^{-*} \bv=0.
$$
Since the first of these equations implies $\bv=\Pi^+ \bv$, we now have
$$
\Pi^+\mathbf b^{-*} \Pi^+\bv=0,
$$
hence by Lemma~\ref{l-low_freq}, $\Pi^+ \bv=0$ and so $\bv=0$, i.e., \eqref{homogpullback} holds.

    Now Lemma \ref{l-existence} yields the existence and uniqueness of $u\in \dot H^1_{\mathrm{loc}}$ such that $u$ solves the homogeneous equation \eqref{stationary} and $\vl$, $\vr$ as the Neumann data on $\jl$, $\jr$ respectively.

    Rewriting the map \eqref{qlr} from incoming to outgoing data $\bscd^o=\bscd_L^o+\bscd_R^o$,
    we consequently obtain an expression for the scattering matrix $\mathbf S$ in the form
    \begin{equation}\label{scattering2} \mathbf S = \begin{pmatrix} \mathrm{Id} & \mathrm{Id} \end{pmatrix}\mathbf T^{-1}\begin{pmatrix} \Pi^-\mathbf b^*\Pi^- \\ \Pi^+\mathbf b^{-*}\Pi^+ \end{pmatrix}. 
    \end{equation}
    To see the microlocal structure of $\mathbf S$, note that by the calculus of wavefront sets on $\mathbb S^1,$ $\mathbf T$ is of the form $\mathrm{Id} + \mathbf{R}$ with $\mathbf{R}$ a (vector-valued) smoothing operator.  Since smoothing operators form an ideal, the inverse must then be of the same form.  Hence the form of the scattering matrix as well as the mapping property~\eqref{eq:S_map_prop} follows from the definition \eqref{scattering2}.

    Let us now show that $\mathbf S$ is unitary on $\mathring{H}^{-\frac12}$.
    For that we define 
    \[ w_{\bullet}:=\int_0^{\theta}\mathbf v_{\bullet} = \sum_{k\neq 0} \frac{\widehat{v}_{\bullet} (k) }{ik} (e^{ik\theta}-1) \in C^{\infty}(\mathbb T), \ \bullet=\mathrm{L}, \mathrm R. \]
    We compute the flux of $w_{\bullet}$
    \[ \mathbf F(w_{\bullet}) = \frac{1}{i}\int_{\mathbb S^1} \overline{w_{\bullet}} \mathrm d w_{\bullet} = 2\pi \sum_{k\neq 0}\frac{|\widehat{v}_{\bullet}(k)|^2}{k} = 2\pi \left( \|\Pi^+\mathbf v_{\bullet}\|^2_{\mathring{H}^{-\frac12}} - \|\Pi^-\mathbf v_{\bullet}\|^2_{\mathring{H}^{-\frac12}} \right)
    \]
    Since $\vl=\mathbf b^*\vr$, we must have $\mathbf F(w_{\mathrm L}) = \mathbf F(w_{\mathrm R})$. Thus,
    \[ \|\Pi^+\vl\|^2_{\mathring{H}^{-\frac12}} + \|\Pi^-\vr\|^2_{\mathring{H}^{-\frac12}} = \|\Pi^+\vr\|^2_{\mathring{H}^{-\frac12}} + \| \Pi^-\vl \|^2_{\mathring{H}^{-\frac12}}. \]
    This shows that $\mathbf S$ is unitary on $\mathring{H}^{-\frac12}$.
    \end{proof}

\section{Outgoing resolvent}

\subsection{Outgoing solutions}
Let us now construct outgoing solutions to the inhomogeneous problem \eqref{inhomogeneous} and thus obtain the proof of Theorem~\ref{t-outgoing}. We first have the following lemma about the boundary reduced equation~\eqref{e-transport}:
\begin{lemm}\label{l-outdata}
    Suppose $\mathbf g\in \mathring L^2(\mathbb S^1; T^*\mathbb S^1)$. Then there exist $\vl, \vr\in \mathring L^2(\mathbb S^1; T^*\mathbb S^1)$ such that \eqref{e-transport} holds and 
    \[ \Pi^+\vl + \Pi^-\vr=0. \]
\end{lemm}
\begin{proof}
Using Theorem \ref{t-matrix}, we can construct unique $(\vl^0, \vr^0)$ such that $\vl^0-\mathbf b^*\vr^0=0$ and has incoming part given by 
\[(\Pi^+ \vl^0, \Pi^- \vr^0) = (0, -\Pi^+ \mathbf g).\]
Now one can check that $(\vl, \vr) := (\vl^0 + \mathbf g, \vr^0)$ satisfies the conditions of the lemma. 
\end{proof}

The construction of the outgoing solution operator $\mathcal R(\lambda)$ then follows from the above lemma, together with Lemma \ref{l-existence} and the uniqueness from Theorem~\ref{t-matrix}:
\begin{proof}[Proof of Theorem \ref{t-outgoing}]
For $f \in L^2_{\comp}$, define $g \in H^1_{\comp}(\partial \Omega_\uparrow)$ by~\eqref{e-g}, and let $\mathbf g \in \mathring L^2(\mathbb S^1; T^*\mathbb S^1)$ be as in~\eqref{e-vpm}. With such $\mathbf g$, we can construct $\vl, \vr\in \mathring L^2(\mathbb S^1; T^*\mathbb S^1)$ as in Lemma~\ref{l-outdata}. Therefore, it follows from Lemma~\ref{l-existence} that there exists $u \in \dot H^1_{\loc}(\Omega)$ to \eqref{inhomogeneous} with $\vl$ and $\vr$ as the Neumann data on the left and right fundamental intervals respectively. It follows from Lemma~\ref{l-outdata} that 
\begin{equation}\label{e-outproj}
    \Pi^+(\partial_{x_2}u|_{\mathcal J_{\mathrm L}})=\Pi^-( \partial_{x_2} u|_{\mathcal J_{\mathrm R}} )=0.
\end{equation}
Furthermore, it follows from Theorem~\ref{t-matrix} that a solution to~\eqref{inhomogeneous} satisfying~\eqref{e-outproj} is unique. Setting
\[\mathcal R(\lambda)f := u,\]
we have constructed the desired outgoing solution operator. 
\end{proof}

\bibliography{lap}
\bibliographystyle{alpha}

\end{document}